\documentclass{amsart}
\usepackage{amssymb}
\usepackage{amsmath}
\usepackage{amsfonts}
\usepackage{geometry}

\newtheorem{theorem}{Theorem}
\theoremstyle{plain}
\newtheorem{acknowledgement}{Acknowledgement}

\newtheorem{definition}{Definition}
\newtheorem{example}{Example}

\newtheorem{lemma}{Lemma}

\newtheorem{remark}{Remark}

\numberwithin{equation}{section}
\input{tcilatex}

\begin{document}
\title{A note on higher dimensional $p$-variation}
\author{Peter Friz and Nicolas Victoir}

\begin{abstract}
We discuss $p$-variation regularity of real-valued functions defined on $%
\left[ 0,T\right] ^{2}$, based on rectangular increments. When $p>1$, there
are two slightly different notions of $p$-variation; both of which are
useful in the context of Gaussian roug paths. Unfortunately, these concepts
were blurred in previous works \cite{FV07, FVforth}; the purpose of this
note is to show that the afore-mentioned notions of $p$-variations are "$%
\varepsilon $-close". In particular, all arguments relevant for Gaussian
rough paths go through with minor notational changes.
\end{abstract}

\maketitle

\section{Higher-dimensional $p$-variation}

Let $T>0$ and $\Delta _{T}=\left\{ \left( s,t\right) :0\leq s\leq t\leq
T\right\} .$We shall regard $\left( \left( a,b\right) ,\left( c,d\right)
\right) \in \Delta _{T}\times \Delta _{T}$ as (closed) \textbf{rectangle} $%
A\subset \left[ 0,T\right] ^{2}$; 
\begin{equation*}
A:=\left( 
\begin{array}{c}
a,b \\ 
c,d%
\end{array}%
\right) :=\left[ a,b\right] \times \left[ c,d\right] ;
\end{equation*}%
if $a=b$ or $c=d$ we call $A$ degenerate. Two rectangles are called \textbf{%
essentially disjoint }if their intersection is empty or degenerate. A 
\textbf{partition} $\Pi $ of a rectangle $R\subset \left[ 0,T\right] ^{2}$
is then a a finite set of essentially disjoint rectangles\textbf{,} whose
union is $R$; the family of all such partitions is denoted by $\mathcal{P}%
\left( R\right) $. Recall that \textbf{rectangular increments} of a function 
$f:\left[ 0,T\right] ^{2}\rightarrow \mathbb{R}$ are defined in terms of $f$
evaluated at the four corner points of $A$,%
\begin{equation*}
f\left( A\right) :=f\left( 
\begin{array}{c}
a,b \\ 
c,d%
\end{array}%
\right) :=f\left( 
\begin{array}{c}
b \\ 
d%
\end{array}%
\right) -f\left( 
\begin{array}{c}
a \\ 
d%
\end{array}%
\right) -f\left( 
\begin{array}{c}
b \\ 
c%
\end{array}%
\right) +f\left( 
\begin{array}{c}
a \\ 
c%
\end{array}%
\right) .
\end{equation*}%
Let us also say that a \textbf{dissection} $D$ of an interval $\left[ a,b%
\right] \subset \left[ 0,T\right] $ is of the form $D=\left( a=t_{0}\leq
t_{1}\leq \dots \leq t_{n}=b\right) $; we write $\mathcal{D}\left( \left[ a,b%
\right] \right) $ for the family of all such dissections.

\begin{definition}
Let $p\in \lbrack 1,\infty )$. A function $f:\left[ 0,T\right]
^{2}\rightarrow \mathbb{R}$ has finite $p$\textbf{-variation} if%
\begin{equation*}
V_{p}\left( f;\left[ s,t\right] \times \left[ u,v\right] \right) :=\left(
\sup_{\substack{ D=\left( t_{i}\right) \in \mathcal{D}\left( \left[ s,t%
\right] \right)  \\ D^{\prime }=\left( t_{j}^{\prime }\right) \in \mathcal{D}%
\left( \left[ u,v\right] \right) }}\sum_{i,j}\left\vert f\left( 
\begin{array}{c}
t_{i},t_{i+1} \\ 
t_{j}^{\prime },t_{j+1}^{\prime }%
\end{array}%
\right) \right\vert ^{p}\right) ^{\frac{1}{p}}<\infty ;
\end{equation*}%
it has finite \textbf{controlled }$p$\textbf{-variation\footnote{%
Our main theorem below will justify this terminology.}} if 
\begin{equation*}
\left\vert f\right\vert _{p\text{-var};\left[ s,t\right] \times \left[ u,v%
\right] }:=\sup_{\Pi \in \mathcal{P}\left( \left[ s,t\right] \times \left[
u,v\right] \right) }\left( \sum_{A\in \Pi }\left\vert f\left( A\right)
\right\vert ^{p}\right) ^{1/p}<\infty .
\end{equation*}
\end{definition}

The difference is that in the first definition (i.e. of $V_{p}$) the $\sup $
is taken over \textbf{grid-like partitions}, 
\begin{equation*}
\left\{ \left( 
\begin{array}{c}
t_{i},t_{i+1} \\ 
t_{j}^{\prime },t_{j+1}^{\prime }%
\end{array}%
\right) :1\leq i\leq n,1\leq j\leq m\right\} ,
\end{equation*}%
based on $D,D^{\prime }$ where $D=\left( t_{i}:1\leq i\leq n\right) \in 
\mathcal{D}\left( \left[ s,t\right] \right) $ and $D^{\prime }=\left(
t_{j}^{\prime }:1\leq j\leq m\right) \in \mathcal{D}\left( \left[ u,v\right]
\right) $. Clearly, not every partition is grid-like (consider e.g. $\left[
0,2\right] ^{2}=\left[ 0,1\right] ^{2}\cup \left[ 1,2\right] \times \left[
0,1\right] \cup \left[ 0,2\right] \times \left[ 1,2\right] $) hence%
\begin{equation*}
V_{p}\left( f;R\right) \leq \left\vert f\right\vert _{p\text{-var};R}.
\end{equation*}%
for every rectangle $R\subset \left[ 0,T\right] ^{2}$.

\begin{definition}
\label{Def2DControl}A map $\omega :\Delta _{T}\times \Delta _{T}\rightarrow
\lbrack 0,\infty )$ is called \textbf{2D\ control} if it is continuous, zero
on degenerate rectangles, and \textbf{super-additive} in the sense that, for
all rectangles $R\subset \left[ 0,T\right] $,%
\begin{equation*}
\sum_{i=1}^{n}\omega \left( R_{i}\right) \leq \omega \left( R\right) ,\text{
whenever }\left\{ R_{i}:1\leq i\leq n\right\} \in \mathcal{P}\left( R\right)
.
\end{equation*}
\end{definition}

Our result is

\begin{theorem}
\label{thm}(i) For any function $f:\left[ 0,T\right] ^{2}\rightarrow \mathbb{%
R}$ and any rectangle $R\subset \left[ 0,T\right] $,%
\begin{equation}
\left\vert f\right\vert _{1\text{-var;}R}=V_{1}\left( f;R\right) .
\label{thmp1}
\end{equation}%
(ii) Let $p\in \lbrack 1,\infty )$ and $\varepsilon >0$. There exists a
constant $c=c\left( p,\varepsilon \right) \geq 1$ such that, for any
function $f:\left[ 0,T\right] ^{2}\rightarrow \mathbb{R}$ and any rectangle $%
R\subset \left[ 0,T\right] $,%
\begin{equation}
\frac{1}{c\left( p,\varepsilon \right) }\left\vert f\right\vert _{\left(
p+\varepsilon \right) \text{-var;}R}\leq V_{p}\left( f;R\right) \leq
\left\vert f\right\vert _{p\text{-var;}R}.  \label{thmp}
\end{equation}%
(iii) If $f:\left[ 0,T\right] ^{2}\rightarrow \mathbb{R}$ is of finite
controlled $p$-variation, then $R\mapsto \left\vert f\right\vert _{p\text{%
-var};R}^{p}$ is super-additive.\newline
(iv) If $f:\left[ 0,T\right] ^{2}\rightarrow \mathbb{R}$ is continuous and
of finite controlled $p$-variation, then $R\mapsto \left\vert f\right\vert
_{p\text{-var};R}^{p}$ is a 2D control. Thus, in particular, there exists a
2D\ control $\omega $ such that%
\begin{equation*}
\forall \text{ rectangles }R\subset \left[ 0,T\right] :\left\vert f\left(
R\right) \right\vert ^{p}\leq \omega \left( R\right)
\end{equation*}
\end{theorem}

As will be seen explicitly in the following example, there exist functions $%
f $ which are of finite $p$-variation but of infinite controlled $p$%
-variation; that is,%
\begin{equation*}
V_{p}\left( f;\left[ 0,T\right] ^{2}\right) <\left\vert f\right\vert _{p%
\text{-var;}\left[ 0,T\right] }=+\infty
\end{equation*}%
which also shows that one cannot take $\varepsilon =0$ in (\ref{thmp}). In
the same example we see that $p$-variation $R\mapsto V_{p}\left( f;R\right)
^{p}$ can fail to be super-additive\footnote{%
... in contrast to controlled $p$-variation $R\mapsto \left\vert
f\right\vert _{p\text{-var};R}^{p}$ which yields a 2D control, cf part (iv)
of the theorem.}.

\begin{example}[{Finite $\left( 1/2H\right) $-variation of fBM covariance, $%
H\in (0,1/2].$ }]
Let $\beta ^{H}$ denote fractional Brownian motion with Hurst parameter $H$;
its covariance is given by%
\begin{equation*}
C^{H}\left( s,t\right) :=\mathbb{E}\left( \beta _{s}^{H}\beta
_{t}^{H}\right) :=\frac{1}{2}\left( t^{2H}+s^{2H}-\left\vert t-s\right\vert
^{2H}\right) ,\,\,\,\,s,t\in \left[ 0,T\right] ^{2},\,H\in (0,1/2].
\end{equation*}%
We show that $C^{H}$ has finite $1/\left( 2H\right) $-variation in 2D sense%
\footnote{%
This is a minor modification of the argument in \cite{FVforth} where it was
assumed that $D=D^{\prime }$.} and more precisely, 
\begin{equation*}
V_{1/\left( 2H\right) }\left( C^{H};\left[ s,t\right] ^{2}\right) \leq
c_{H}\left\vert t-s\right\vert ^{2H}.
\end{equation*}%
(By fractional scaling it would suffice to consider $\left[ s,t\right] =%
\left[ 0,1\right] $ but this does not simplify the argument which follows).
Consider $D=\left( t_{i}\right) ,D^{\prime }=\left( t_{j}^{\prime }\right)
\in \mathcal{D}\left[ s,t\right] $. Clearly,%
\begin{eqnarray}
3^{1-\frac{1}{2H}}\sum_{j}\left\vert E\left[ \beta _{t_{i},t_{i+1}}^{H}\beta
_{t_{j}^{\prime },t_{j+1}^{\prime }}^{H}\right] \right\vert ^{\frac{1}{2H}}
&\leq &3^{1-\frac{1}{2H}}\left\vert E\left[ \beta _{t_{i},t_{i+1}}^{H}\beta
_{\cdot }^{H}\right] \right\vert _{\frac{1}{2H}\text{-var;}\left[ s,t\right]
}^{\frac{1}{2H}}  \notag \\
&\leq &\left\vert E\left[ \beta _{t_{i},t_{i+1}}^{H}\beta _{\cdot }^{H}%
\right] \right\vert _{\frac{1}{2H}\text{-var;}\left[ s,t_{i}\right] }^{\frac{%
1}{2H}}  \label{First} \\
&&+\left\vert E\left[ \beta _{t_{i},t_{i+1}}^{H}\beta _{\cdot }^{H}\right]
\right\vert _{\frac{1}{2H}\text{-var;}\left[ t_{i},t_{i+1}\right] }^{\frac{1%
}{2H}}  \label{Middle} \\
&&+\left\vert E\left[ \beta _{t_{i},t_{i+1}}^{H}\beta _{\cdot }^{H}\right]
\right\vert _{\frac{1}{2H}\text{-var;}\left[ t_{i+1},t\right] }^{\frac{1}{2H}%
},  \label{Last}
\end{eqnarray}%
by super-additivity of (1D!) controls. The middle term (\ref{Middle}) is
estimated by 
\begin{eqnarray*}
\left\vert E\left[ \beta _{t_{i},t_{i+1}}^{H}\beta _{\cdot }^{H}\right]
\right\vert _{\frac{1}{2H}\text{-var;}\left[ t_{i},t_{i+1}\right] }^{\frac{1%
}{2H}} &=&\sup_{\left( s_{k}\right) \in \mathcal{D}\left[ t_{i},t_{i+1}%
\right] }\sum_{k}\left\vert E\left[ \beta _{t_{i},t_{i+1}}^{H}\beta
_{s_{k},s_{k+1}}^{H}\right] \right\vert ^{\frac{1}{2H}} \\
&\leq &c_{H}\left\vert t_{i+1}-t_{i}\right\vert ,
\end{eqnarray*}%
where we used that $\left[ s_{k},s_{k+1}\right] \subset \left[ t_{i},t_{i+1}%
\right] $ implies $\left\vert E\left[ \beta _{t_{i},t_{i+1}}^{H}\beta
_{s_{k},s_{k+1}}^{H}\right] \right\vert \leq c_{H}\left\vert
s_{k+1}-s_{k}\right\vert ^{2H}$. The first term (\ref{First}) and the last
term (\ref{Last}) are estimated by exploiting the fact that disjoint
increments of fractional Brownian motion have negative correlation when $%
H<1/2$ (resp. zero correlation in the Brownian case, $H=1/2$); that is, $%
E\left( \beta _{c,d}^{H}\beta _{a,b}^{H}\right) \leq 0$ whenever $a\leq
b\leq c\leq d$. We can thus estimate (\ref{First}) as follows; 
\begin{eqnarray*}
\left\vert E\left[ \beta _{t_{i},t_{i+1}}^{H}\beta _{\cdot }^{H}\right]
\right\vert _{\frac{1}{2H}\text{-var;}\left[ s,t_{i}\right] }^{\frac{1}{2H}}
&=&\left\vert E\left[ \beta _{t_{i},t_{i+1}}^{H}\beta _{s,t_{i}}^{H}\right]
\right\vert ^{\frac{1}{2H}} \\
&\leq &2^{\frac{1}{2H}-1}\left( \left\vert E\left[ \beta
_{t_{i},t_{i+1}}^{H}\beta _{s,t_{i}}^{H}\right] \right\vert ^{\frac{1}{2H}}+E%
\left[ \left\vert \beta _{t_{i},t_{i+1}}^{H}\right\vert ^{2}\right] ^{\frac{1%
}{2H}}\right) .
\end{eqnarray*}%
The covariance of fractional Brownian motion gives immediately $E\left[
\left\vert \beta _{t_{i},t_{i+1}}^{H}\right\vert ^{2}\right] ^{\frac{1}{2H}%
}=c_{H}\left( t_{i+1}-t_{i}\right) $. On the other hand, $\left[
t_{i},t_{i+1}\right] \subset \left[ s,t_{i+1}\right] $ implies $\left\vert E%
\left[ \beta _{t_{i},t_{i+1}}^{H}\beta _{s,t_{i}}^{H}\right] \right\vert ^{%
\frac{1}{2H}}\leq c_{H}\left\vert t_{i+1}-t_{i}\right\vert $; hence%
\begin{equation*}
\left\vert E\left[ \beta _{t_{i},t_{i+1}}^{H}\beta _{\cdot }^{H}\right]
\right\vert _{\frac{1}{2H}\text{-var;}\left[ s,t_{i}\right] }^{\frac{1}{2H}%
}\leq c_{H}\left\vert t_{i+1}-t_{i}\right\vert \text{.}
\end{equation*}%
As already remarked, the last term is estimated similarly. It only remains
to sum up and to take the supremum over all dissections $D$ and $D^{\prime }$%
.
\end{example}

\begin{example}[Failure of super-addivity of $\left( 1/2H\right) $%
-variation, infinite controlled $\left( 1/2H\right) $-variation of fBM
covariance, $H\in (0,1/2).$ ]
We saw above that%
\begin{equation*}
V_{1/\left( 2H\right) }\left( C^{H};\left[ 0,T\right] ^{2}\right) <\infty .
\end{equation*}%
When $H=1/2$ we deal with Brownian motion and see that its covariance has
finite $1$-variation, which, by (i),(iv) of theorem \ref{thm}, constitues a
2D\ control for $C^{1/2}$. In contrast, we claim that, for $H<1/2$, there
does \underline{not} exist a 2D control for the $1/\left( 2H\right) $%
-variation of $C^{H}$. In fact, the sheer existence of a super-additive map $%
\omega $ (in the sense of definition \ref{Def2DControl}) such that%
\begin{equation*}
\forall \text{ rectangles }R\subset \left[ 0,T\right] :\left\vert
C^{H}\left( R\right) \right\vert ^{1/\left( 2H\right) }\leq \omega \left(
R\right)
\end{equation*}%
leads to a contradiction as follows: assume that such a $\omega $ exists. By
super-addivity,%
\begin{equation*}
\bar{\omega}\left( R\right) :=\left\vert C^{H}\right\vert _{1/\left(
2H\right) \text{-var;}R}^{1/\left( 2H\right) }\leq \omega \left( R\right)
<\infty
\end{equation*}%
and $\bar{\omega}$ is super-additive (in fact, a 2D control) thanks to part
(iv) of the theorem. On the other hand, by fractional scaling there exists $%
C $ such that%
\begin{equation*}
\forall \left( s,t\right) \in \Delta _{T}:\bar{\omega}\left(
[s,t]^{2}\right) =C\left\vert t-s\right\vert .
\end{equation*}%
Let us consider the case $T=2$ and the partition%
\begin{equation*}
\lbrack 0,2]^{2}=[0,1]^{2}\cup \lbrack 1,2]^{2}\cup R\cup R^{\prime }
\end{equation*}%
with $R=[0,1]\times \lbrack 1,2],\,R^{\prime }=[1,2]\times \lbrack 0,1]$.
Super-addivitiy of $\bar{\omega}$ gives%
\begin{eqnarray*}
\bar{\omega}\left( [0,1]^{2}\right) +\bar{\omega}\left( [1,2]^{2}\right) +%
\bar{\omega}\left( R\right) +\bar{\omega}\left( R^{\prime }\right) &\leq &%
\bar{\omega}\left( [0,2]^{2}\right) , \\
C\left( 1-0\right) +C\left( 2-1\right) +\bar{\omega}\left( R\right) +\bar{%
\omega}\left( R^{\prime }\right) &\leq &2C,
\end{eqnarray*}%
hence $\bar{\omega}\left( R\right) =\bar{\omega}\left( R^{\prime }\right) =0$%
, and thus also%
\begin{equation*}
C^{H}\left( R\right) =\mathbb{E}\left[ \left( B_{1}^{H}-B_{0}^{H}\right)
\left( B_{2}^{H}-B_{1}^{H}\right) \right] =0;
\end{equation*}%
which is false for $H\neq 1/2$ and hence the desired contradiction. En
passant, we see that we must have%
\begin{equation*}
\left\vert C^{H}\right\vert _{1/\left( 2H\right) \text{-var;}\left[ 0,T%
\right] ^{2}}=+\infty ;
\end{equation*}%
for otherwise part (iv) of theorem \ref{thm} would yield a 2D control for
the $1/\left( 2H\right) $-variation of $C^{H}$. This also shows that, with $%
f=C^{H}$ and $p=1/\left( 2H\right) $ one has%
\begin{equation*}
V_{p}\left( f;\left[ 0,T\right] ^{2}\right) <\left\vert f\right\vert _{p%
\text{-var;}\left[ 0,T\right] ^{2}}=+\infty .
\end{equation*}
\end{example}

\begin{remark}
The previous examples clearly show the need for theorem \ref{thm};
variational regularity of $C^{H}$ can be controlled upon considering $\left[
\left( 1/2H\right) +\varepsilon \right] $-variation rather than $1/\left(
2H\right) $-variation. In applications, this distinction never matters.\
Existence for Gaussian rough paths for instance, requires $1/\left(
2H\right) <2$ and one can always insert a small enough $\varepsilon .$ It
should also be point out that, by fractional scaling,%
\begin{equation*}
\left\vert C^{H}\right\vert _{\left[ 1/\left( 2H\right) +\varepsilon \right] 
\text{-var;}\left[ s,t\right] ^{2}}\propto \left\vert t-s\right\vert ^{2H};
\end{equation*}%
hence, even in estimates that involve directly that variational regularity
of $C^{H}$, no $\varepsilon $ loss is felt.
\end{remark}

\begin{acknowledgement}
The authors are indebted to Bruce Driver for pointing out, in the most
constructive and gentle way, that $R\mapsto V_{p}\left( f;R\right) ^{p}$ is
not, in general, super-additive.
\end{acknowledgement}

\section{Proof of (i)}

We claim the \textit{controlled} $1$-variation is exactly equal to its $1$%
-variation. More precisely, for all rectangles $R\subset \left[ 0,T\right]
^{2}$ we have 
\begin{equation*}
\left\vert f\right\vert _{1\text{-var;}R}=V_{1}\left( f;R\right) .
\end{equation*}

\begin{proof}
Trivially $V_{1}\left( f;R\right) \leq \left\vert f\right\vert _{1\text{-var;%
}R}$. For the other inequality, assume $\Pi $ is a partition of $R$. It is
obvious that one can find a grid-like partition $\tilde{\Pi}$, based on $%
D\times D^{\prime }$, for sufficiently fine dissections $D,D^{\prime }$,
which \textbf{refines} $\Pi $ in the sense that every $A\in \Pi $ can be
expressed as%
\begin{equation*}
A=\cup _{i}A_{i}\text{ (essentially disjoint), }A_{i}\in \tilde{\Pi}.
\end{equation*}%
From the very definition of rectangular increments, we have $f\left(
A\right) =\sum_{i}f\left( A_{i}\right) $ and it follows that $\left\vert
f\left( A\right) \right\vert \leq \sum_{i}\left\vert f\left( A_{i}\right)
\right\vert $. (Note that this estimate is false if $\left\vert \cdot
\right\vert $ is replaced by $\left\vert \cdot \right\vert ^{p},p>1$.) Hence%
\begin{equation*}
\sum_{A\in \Pi }\left\vert f\left( A\right) \right\vert \leq \sum_{A\in 
\tilde{\Pi}}\left\vert f\left( A\right) \right\vert \leq \left\vert
f\right\vert _{1\text{-var;}R}.
\end{equation*}%
It now suffices to take the supremum over all such $\Pi $ to see that $%
\left\vert f\right\vert _{1\text{-var;}R}\leq V_{1}\left( f;R\right) $.
\end{proof}

\section{Proof of (ii)}

The second inequality $V_{p}\left( f;R\right) \leq \left\vert f\right\vert
_{p\text{-var;}R}$ is trivial. Furthermore, if $V_{p}\left( f;R\right)
=+\infty $ there is nothing to show so we may assume $V_{p}\left( f;R\right)
<+\infty $. We claim that, for all rectangle $R\subset \left[ 0,T\right]
^{2} $, 
\begin{equation*}
\left\vert f\right\vert _{p+\varepsilon \text{-var;}R}\leq c\left(
p,\varepsilon \right) V_{p}\left( f;R\right) \text{.}
\end{equation*}%
For the proof we note first that there is no loss in generality in taking $R=%
\left[ 0,T\right] ^{2}$; an affine reparametrization of each axis will
transform $R$ into $\left[ 0,T\right] ^{2}$, while leaving all rectangular
increments invariant. The plan is to show, for an arbitrary partition $%
\left( Q_{k}\right) \in \mathcal{P}\left( \left[ 0,T\right] ^{2}\right) $,
the estimate%
\begin{equation*}
\left( \sum_{k}\left\vert f\left( Q_{k}\right) \right\vert ^{p+\varepsilon
}\right) ^{\frac{1}{p+\varepsilon }}\leq c\left( p,\varepsilon \right)
V_{p}\left( f;\left[ 0,T\right] ^{2}\right) .
\end{equation*}%
where $c$ depends only on $p,\varepsilon $ for any partition $\left(
Q_{k}\right) \in \mathcal{P}\left( \left[ 0,T\right] ^{2}\right) $. The key
observation is that for a suitable choice of $y,x,D=\left( t_{i}\right)
,D^{\prime }=\left( t_{j}^{\prime }\right) $ we have%
\begin{eqnarray}
\sum_{k}\left\vert f\left( Q_{k}\right) \right\vert ^{p+\varepsilon }
&=&\sum_{k}\left\vert f\left( Q_{k}\right) \right\vert ^{p+\varepsilon
-1}sgn\left( f\left( Q_{k}\right) \right) f\left( Q_{k}\right)
\label{RHSintuit} \\
&=&\sum_{i}\sum_{j}y\left( 
\begin{array}{c}
t_{i} \\ 
t_{j}^{\prime }%
\end{array}%
\right) x\left( 
\begin{array}{c}
t_{i-1},t_{i} \\ 
t_{j-1}^{\prime },t_{j}^{\prime }%
\end{array}%
\right)  \notag \\
&=&:\int_{D\times D^{\prime }}y\,dx.  \notag
\end{eqnarray}%
Indeed, we may take (as in the proof of part (i)) sufficiently fine
dissections $D=\left( t_{i}\right) ,D^{\prime }=\left( t_{j}^{\prime
}\right) \in $ $\mathcal{D}\left[ 0,T\right] $ such that the grid-like
partition based on $D\times D^{\prime }$ refines $\left( Q_{k}\right) $;
followed by setting\footnote{%
The "right-closed" form of $\hat{Q}_{k}$ in the definition of $y$ is tied to
our definition of $\int_{D\times D^{\prime }}y\,dx$ which imposes
"right-end-point-evaluation" of $y$. Recall also that $Q_{k}$ is \textit{%
really} a point in $\left( \left( a,b\right) ,\left( c,d\right) \right) \in $
$\Delta _{T}\times \Delta _{T};$ viewing it as closed rectangle is pure
convention.\ }%
\begin{eqnarray*}
x &:&=f \\
y &:&=\sum_{k}\left\vert f\left( Q_{k}\right) \right\vert ^{p-1+\varepsilon
}sgn\left( f\left( Q_{k}\right) \right) \mathbb{I}_{\hat{Q}_{k}}
\end{eqnarray*}%
where $\hat{Q}_{k}$ is the of the form $(a,b]\times (c,d]$ whenever $Q_{k}=%
\left[ a,b\right] \times \left[ c,d\right] $. Lemma \ref{CrucialLemmaNew}
below, applied with $p+\varepsilon $ instead of $p$, says%
\begin{equation*}
V_{q}\left( y;\left[ 0,T\right] ^{2}\right) \leq 4\left\vert
\sum_{k}\left\vert x\left( Q_{k}\right) \right\vert ^{p+\varepsilon
}\right\vert ^{\frac{1}{q}}
\end{equation*}%
where $q:=1/\left( 1-1/\left( p+\varepsilon \right) \right) $ denotes the H%
\"{o}lder conjugate of $p+\varepsilon $. Since%
\begin{equation*}
\frac{1}{p}+\frac{1}{q}=1+\left( \frac{1}{p}-\frac{1}{p+\varepsilon }\right)
>1,
\end{equation*}%
noting also that $y\left( 0,\cdot \right) =y\left( \cdot ,0\right) =0$, we
can use \textbf{Young-Towghi's maximal inequality} \cite[Thm 2.1.]{To},
included for the reader's convenience as theorem \ref{YTMI} in the appendix,
to obtain the estimate%
\begin{eqnarray*}
\sum_{k}\left\vert f\left( Q_{k}\right) \right\vert ^{p+\varepsilon } &\leq
&c\left( p,\varepsilon \right) V_{q}\left( y;\left[ 0,T\right] ^{2}\right)
V_{p}\left( x;\left[ 0,T\right] ^{2}\right) \\
&\leq &4c\left( p,\varepsilon \right) \left\vert \sum_{k}\left\vert x\left(
Q_{k}\right) \right\vert ^{p+\varepsilon }\right\vert ^{\frac{1}{q}%
}V_{p}\left( x;\left[ 0,T\right] ^{2}\right)
\end{eqnarray*}%
Since $1-\frac{1}{q}=\frac{1}{p+\varepsilon }$ and $x=f$ we see that%
\begin{equation*}
\left( \sum_{k}\left\vert f\left( Q_{k}\right) \right\vert ^{p+\varepsilon
}\right) ^{\frac{1}{p+\varepsilon }}\leq 4c\left( p,\varepsilon \right)
V_{p}\left( f;\left[ 0,T\right] ^{2}\right)
\end{equation*}%
and conclude by taking the supremum over all partitions $\left( Q_{k}\right)
\in \mathcal{P}\left( \left[ 0,T\right] ^{2}\right) $.

\begin{lemma}
\label{CrucialLemmaNew}Fix $p\geq 1$ and write $p^{\prime }$ for the H\"{o}%
lder conjugate i.e. $1/p^{\prime }+1/p=1$. Let $\left( Q_{j}\right) \in 
\mathcal{P}\left( \left[ 0,T\right] ^{2}\right) $ and $y=\sum_{j}\left\vert
x\left( Q_{j}\right) \right\vert ^{p-1}sgn\left( x\left( Q_{j}\right)
\right) \mathbb{I}_{\hat{Q}_{j}}$. Then%
\begin{equation*}
V_{p^{\prime }}\left( y,\left[ 0,T\right] ^{2}\right) \leq \left\vert
y\right\vert _{p^{\prime }\text{-var;}\left[ 0,T\right] ^{2}}\leq 4\left(
\sum_{i}\left\vert x\left( Q_{i}\right) \right\vert ^{p}\right)
^{1/p^{\prime }}.
\end{equation*}
\end{lemma}

\begin{proof}
Only the second inequality requires a proof. By definition, $\left(
Q_{j}\right) $ forms a partition of $\left[ 0,T\right] ^{2}$ into
essentially disjoint rectangles and we note that $y\left( .,0\right)
=y\left( 0,.\right) =0$. Consider now another partition $\left( R_{i}\right)
\in \mathcal{P}\left( \left[ 0,T\right] ^{2}\right) $. The rectangular
increments of $y$ over $R_{i}$ spells out as "$+--+$ sum" of $y$ evaluated
at the corner points of $R_{i}$. Recall that on each set $\hat{Q}_{j}$ the
function $y~$takes the consant value%
\begin{equation*}
c_{j}:=\left\vert x\left( Q_{j}\right) \right\vert ^{p-1}sgn\left( x\left(
Q_{j}\right) \right) .
\end{equation*}%
Since the corner points of $R_{i}$ are elements of $Q_{j_{1}}\cup
Q_{j_{2}}\cup Q_{j_{3}}\cup Q_{j_{4}}$ for suitable (not necessarily
distinct) indices $j_{1},\dots ,j_{4}$ we clearly have the (crude) estimate%
\begin{equation}
\left\vert y\left( R_{i}\right) \right\vert \leq \sum_{j\in \left\{
j_{1},j_{2},j_{3},j_{4}\right\} }\left\vert c_{j}\right\vert
\label{CrudeEstimate}
\end{equation}%
and, trivially, any $j\notin \left\{ j_{1},j_{2},j_{3},j_{4}\right\} $ is
not required in estimating $\left\vert y\left( R_{i}\right) \right\vert $.
Let us distinguish a few cases where we can do better than in \ref%
{CrudeEstimate}. \newline
\textbf{Case 1:} There exists $j$ such that all four corner points of $R_{i}$
are elements of $Q_{j}$ (equivalently: $\exists j:R_{i}\subset \hat{Q}_{j}$%
). In this case%
\begin{equation*}
y\left( R_{i}\right) =c_{j}-c_{j}-c_{j}+c_{j}=0.
\end{equation*}%
In particular, such an index $j$ is not required to estimate $\left\vert
y\left( R_{i}\right) \right\vert $.$\newline
$\textbf{Case 2:} There exists $j$ such that precisely two corner points%
\footnote{%
The case that three corner points of $R_{i}$ are elements of $Q_{j}$ already
implies (rectangles!) that all four corner points of $R_{i}$ are elements of 
$Q_{j}$. This is covered by Case 1.} of $R_{i}$ are elements of $Q_{j}$. It
follows that the corner points of $R_{i}$ are elements of $Q_{j_{1}}\cup
Q_{j_{2}}\cup Q_{j}$ for suitable (not necessarily distinct) indices $%
j_{1},j_{2}$. Note however that $j\notin \left\{ j_{1},j_{2}\right\} $. In
this case%
\begin{equation*}
y\left( R_{i}\right) =c_{j_{1}}-c_{j_{2}}-c_{j}+c_{j}=c_{j_{1}}-c_{j_{2}}.
\end{equation*}%
In general, this quantity is non-zero (although it is zero when $j_{1}=j_{2}$%
, which is tantamount to say that $R_{i}\subset Q_{j_{1}}\cup Q_{j}$). Even
so, we note that 
\begin{equation*}
\left\vert y\left( R_{i}\right) \right\vert \leq \left\vert
c_{j_{1}}\right\vert +\left\vert c_{j_{2}}\right\vert
\end{equation*}%
and again the index $j$ is not required in order to estimate $\left\vert
y\left( R_{i}\right) \right\vert $.\newline
\textbf{Case 3:} There exists $j$ such that precisely one corner point of $%
R_{i}$ is an element of $Q_{j}$. In this case, for suitable (not necessarily
distinct) indices $j_{1},j_{2},j_{3}$ with $j\notin \left\{
j_{1},j_{2},j_{3}\right\} $ 
\begin{equation*}
\left\vert y\left( R_{i}\right) \right\vert =\left\vert
c_{j_{1}}-c_{j_{2}}-c_{j_{3}}+c_{j}\right\vert \leq \left\vert
c_{j_{1}}-c_{j_{2}}-c_{j_{3}}\right\vert +\left\vert c_{j}\right\vert .
\end{equation*}%
In this case, the index $j$ is required to estimate $\left\vert y\left(
R_{i}\right) \right\vert $. (There is still the possibily for cancellation
between the other terms. If $j_{2}=j_{3}$ for instance, then $\left\vert
y\left( R_{i}\right) \right\vert \leq \left\vert c_{j_{1}}\right\vert
+\left\vert c_{j}\right\vert $ and indices $j_{2},j_{3}$ are not required;
this corresponds precisely to case 2 applied to $Q_{j_{2}}$. Another
possiblility is that $\left\{ j_{1},j_{2},j_{3}\right\} $ are all distinct
in which case $\left\vert y\left( R_{i}\right) \right\vert \leq \left\vert
c_{j_{1}}\right\vert +\left\vert c_{j_{2}}\right\vert +\left\vert
c_{j_{3}}\right\vert +\left\vert c_{j}\right\vert $ is the best estimate and
all four indices $j_{1},j_{2},j_{3},j$ are needed in the estimate.\newline
The moral of this case-by-case consideration is that only those $j\in \phi
\left( i\right) $ where%
\begin{equation*}
\phi \left( i\right) :=\left\{ j:\text{ precisely one corner point of }R_{i}%
\text{ is an element of }Q_{j}\right\}
\end{equation*}%
are required in estimating $\left\vert y\left( R_{i}\right) \right\vert $;
more precisely,%
\begin{equation*}
\left\vert y\left( R_{i}\right) \right\vert \leq \sum_{j\in \phi \left(
i\right) }\left\vert c_{j}\right\vert .
\end{equation*}%
Since rectangles (here: $R_{i}$)\ have four corner points it is clear that $%
\#\phi \left( i\right) \leq 4$ where $\#$ denotes the cardinality of a set.
Hence%
\begin{equation*}
\left\vert y\left( R_{i}\right) \right\vert ^{p^{\prime }}\leq 4^{p^{\prime
}-1}\sum_{j\in \phi \left( i\right) }\left\vert c_{j}\right\vert ^{p^{\prime
}}\equiv 4^{p^{\prime }-1}\sum_{j}\phi _{i,j}\left\vert c_{j}\right\vert
^{p^{\prime }}
\end{equation*}%
where we introdudced the matrix $\phi _{i,j}$ with value $1$ if $j\in \phi
\left( i\right) $ and zero else. This allows us to write%
\begin{eqnarray*}
\sum_{i}\left\vert y\left( R_{i}\right) \right\vert ^{p^{\prime }} &\leq
&4^{p^{\prime }-1}\sum_{i}\sum_{j}\phi _{i,j}\left\vert c_{j}\right\vert
^{p^{\prime }} \\
&=&4^{p^{\prime }-1}\sum_{j}\left\vert c_{j}\right\vert ^{p^{\prime
}}\sum_{i}\phi _{i,j}.
\end{eqnarray*}%
Consider now, for fixed $j$, the number of rectangles $R_{i}$ which have
precisely one corner point inside $Q_{j}$. Obviously, there can be a most $4$
rectangles with this property. Hence%
\begin{equation*}
\sum_{i}\phi _{i,j}=\#\left\{ i:j\in \phi \left( i\right) \right\} \leq 4.
\end{equation*}%
It follows that%
\begin{equation*}
\sum_{i}\left\vert y\left( R_{i}\right) \right\vert ^{p^{\prime }}\leq
4^{p^{\prime }}\sum_{j}\left\vert c_{j}\right\vert ^{p^{\prime
}}=4^{p^{\prime }}\sum_{j}\left\vert x\left( Q_{j}\right) \right\vert
^{\left( p-1\right) p^{\prime }}=4^{p^{\prime }}\sum_{j}\left\vert x\left(
Q_{j}\right) \right\vert ^{p},
\end{equation*}%
where we used that $\left( p-1\right) p^{\prime }=p$. Upon taking the
supremum over all partitions $\left( R_{i}\right) $ of $\left[ 0,T\right]
^{2}$ we obtain%
\begin{equation*}
\left\vert y\right\vert _{p^{\prime }\text{-var;}\left[ 0,T\right]
^{2}}^{p^{\prime }}\leq 4^{p^{\prime }}\sum_{i}\left\vert x\left(
Q_{i}\right) \right\vert ^{p},
\end{equation*}%
as desired. The proof is finished.\ 
\end{proof}

\section{Proof of (iii)}

The claim is \underline{super-additivity} of%
\begin{equation*}
R\mapsto \sup_{\Pi \in \mathcal{P}\left( R\right) }\sum_{A\in \Pi
}\left\vert f\left( A\right) \right\vert ^{p}.
\end{equation*}
Assume $\left\{ R_{i}:1\leq i\leq n\right\} $ constitutes a partition of $R$%
. Assume also that $\Pi _{i}$ is a partition of $R_{i}$ for every $1\leq
i\leq n$. Clearly, $\Pi :=\cup _{i=1}^{n}\Pi _{i}$ is a partition of $R$ and
hence%
\begin{equation*}
\sum_{i=1}^{n}\sum_{A\in \Pi _{i}}\left\vert f\left( A\right) \right\vert
^{p}=\sum_{A\in \Pi }\left\vert f\left( A\right) \right\vert ^{p}\leq \omega
\left( R\right)
\end{equation*}%
Now taking the supremum over each of the $\Pi _{i}$ gives the desired result.%
\newline

\section{Proof of (iv)}

The assumption is that $f:\left[ 0,T\right] ^{2}\rightarrow \mathbb{R}$ is
continuous and of finite controlled $p$-variation. From (iii),%
\begin{equation*}
\omega \left( R\right) :=\left\vert f\right\vert _{p\text{-var;}R}^{p}
\end{equation*}%
is super-additive as function of $R$. It is also clear that $\omega $ is
zero on degenerate rectangles. It remains to be seen that $\omega :$ $\Delta
_{T}\times \Delta _{T}\rightarrow \lbrack 0,\infty )$ is continuous.

\begin{lemma}
\label{almostSubAdditive}Consider the two (adjacent) rectangles $\left[ a,b%
\right] \times \left[ s,t\right] $ and $\left[ a,b\right] \times \left[ t,u%
\right] $ in $\left[ 0,T\right] ^{2}.$Then,%
\begin{eqnarray*}
\omega \left( 
\begin{array}{c}
a,b \\ 
s,u%
\end{array}%
\right) &\leq &\omega \left( 
\begin{array}{c}
a,b \\ 
s,t%
\end{array}%
\right) +\omega \left( 
\begin{array}{c}
a,b \\ 
t,u%
\end{array}%
\right) \\
&&+p2^{p-1}\omega \left( 
\begin{array}{c}
a,b \\ 
s,u%
\end{array}%
\right) ^{1-1/p}\min \left\{ \omega \left( 
\begin{array}{c}
a,b \\ 
t,u%
\end{array}%
\right) ,\omega \left( 
\begin{array}{c}
a,b \\ 
s,t%
\end{array}%
\right) \right\} ^{1/p}.
\end{eqnarray*}
\end{lemma}

\begin{proof}
From the very definition of $\omega \left( \left[ a,b\right] \times \left[
s,u\right] \right) $, it follows that for every fixed $\varepsilon >0$,
there exists a rectangular (not necessarily grid-like) partition of $\left[
a,b\right] \times \left[ s,u\right] $, say $\Pi \in \mathcal{P}\left( \left[
a,b\right] \times \left[ s,u\right] \right) $, such that%
\begin{equation*}
\sum_{R\in \Pi }\left\vert f\left( R\right) \right\vert ^{p}>\omega \left( 
\begin{array}{c}
a,b \\ 
s,u%
\end{array}%
\right) -\varepsilon .
\end{equation*}%
Let us divide $\Pi $ in $\Pi _{l}\cup \Pi _{m}\cup \Pi _{r}$ where $\Pi _{l}$
contains all $R\in \Pi $ such that $R\subset \left[ a,b\right] \times \left[
s,t\right] $, $\Pi _{r}$ contains all $R\in \Pi :R\subset \left[ a,b\right]
\times \left[ t,u\right] $ and $\Pi _{m}$ contains all remaining rectangles
of $\Pi $ (i.e. the one such that their interior intersect with the line $%
\left[ a,b\right] \times \left[ t,t\right] $. It follows that%
\begin{equation*}
\sum_{R\in \Pi _{l}}\left\vert f\left( R\right) \right\vert ^{p}+\sum_{R\in
\Pi _{m}}\left\vert f\left( R\right) \right\vert ^{p}+\sum_{R\in \Pi
_{r}}\left\vert f\left( R\right) \right\vert ^{p}>\omega \left( 
\begin{array}{c}
a,b \\ 
s,u%
\end{array}%
\right) -\varepsilon
\end{equation*}%
Every $R\in \Pi _{m}$ can be split into (essentially disjoint) rectangles $%
R_{1}\subset \left[ a,b\right] \times \left[ s,t\right] $ and $R_{2}\subset %
\left[ a,b\right] \times \left[ t,u\right] $. Set $\Pi _{m}^{1}=\left\{
R_{1}:R_{1}\in \Pi _{m}\right\} $ and $\Pi _{m}^{2}$ similarly. Note that $%
\Pi _{l}\cup \Pi _{m}^{1}\in \mathcal{P}\left( \left[ a,b\right] \times %
\left[ s,t\right] \right) $ and $\Pi _{m}^{2}\cup \Pi _{r}\in \mathcal{P}%
\left( \left[ a,b\right] \times \left[ t,u\right] \right) $. Then, with%
\begin{equation*}
\Delta :=\sum_{R\in \Pi _{m}}\left[ \left\vert f\left( R\right) \right\vert
^{p}-\left\vert f\left( R_{1}\right) \right\vert ^{p}-\left\vert f\left(
R_{2}\right) \right\vert ^{p}\right]
\end{equation*}%
we have%
\begin{equation*}
\sum_{R\in \Pi _{l}\cup \Pi _{m}^{1}}\left\vert f\left( R\right) \right\vert
^{p}+\sum_{R\in \Pi _{m}^{2}\cup \Pi _{r}}\left\vert f\left( R\right)
\right\vert ^{p}+\Delta >\omega \left( \left[ a,b\right] \times \left[ s,u%
\right] \right) -\varepsilon
\end{equation*}%
and hence $,$we have%
\begin{equation*}
\omega \left( 
\begin{array}{c}
a,b \\ 
s,t%
\end{array}%
\right) +\omega \left( 
\begin{array}{c}
a,b \\ 
t,u%
\end{array}%
\right) +\Delta >\omega \left( 
\begin{array}{c}
a,b \\ 
s,u%
\end{array}%
\right) -\varepsilon \text{.}
\end{equation*}%
We now bound $\Delta .$ As $f(R)=f\left( R_{1}\right) +f\left( R_{2}\right) $%
,%
\begin{eqnarray*}
\Delta &=&\sum_{R^{j}\in \Pi _{m}}\left\vert f\left( R_{1}^{j}\right)
+f\left( R_{2}^{j}\right) \right\vert ^{p}-\left\vert f\left(
R_{1}^{j}\right) \right\vert ^{p}-\left\vert f\left( R_{2}^{j}\right)
\right\vert ^{p} \\
&\leq &\sum_{R\in \Pi _{m}}\left( \left\vert f\left( R_{1}^{j}\right)
\right\vert +\left\vert f\left( R_{2}^{j}\right) \right\vert \right)
^{p}-\left\vert f\left( R_{1}^{j}\right) \right\vert ^{p}-\left\vert f\left(
R_{2}^{j}\right) \right\vert ^{p}. \\
&\leq &\sum_{R\in \Pi _{m}}\left( \left\vert f\left( R_{1}^{j}\right)
\right\vert +\left\vert f\left( R_{2}^{j}\right) \right\vert \right)
^{p}-\left\vert f\left( R_{1}^{j}\right) \right\vert ^{p}
\end{eqnarray*}%
If $R^{j}=[\tau _{j},\tau _{j+1}]\times \left[ c,d\right] ,$ define $%
R_{3}^{j}=[\tau _{j},\tau _{j+1}]\times \left[ s,u\right] .$ Then, quite
obviously, we have $\left\vert f\left( R_{1}^{j}\right) \right\vert ^{p}\leq
\omega \left( R_{3}^{j}\right) $ and $\left\vert f\left( R_{2}^{j}\right)
\right\vert ^{p}\leq \omega \left( R_{3}^{j}\right) .$ By the mean value
theorem, there exists $\theta \in \left[ 0,1\right] $ such that%
\begin{eqnarray*}
&&\left( \left\vert f\left( R_{1}^{j}\right) \right\vert +\left\vert f\left(
R_{2}^{j}\right) \right\vert \right) ^{p}-\left\vert f\left(
R_{1}^{j}\right) \right\vert ^{p} \\
&=&p\left( \left\vert f\left( R_{1}^{j}\right) \right\vert +\theta
\left\vert f\left( R_{2}^{j}\right) \right\vert \right) ^{p-1}\left\vert
f\left( R_{2}^{j}\right) \right\vert \\
&\leq &p2^{p-1}\omega \left( R_{3}^{j}\right) ^{1-1/p}\left\vert f\left(
R_{2}^{j}\right) \right\vert \\
&\leq &p2^{p-1}\omega \left( 
\begin{array}{c}
\tau _{j},\tau _{j+1} \\ 
s,u%
\end{array}%
\right) ^{1-1/p}\omega \left( 
\begin{array}{c}
\tau _{j},\tau _{j+1} \\ 
t,u%
\end{array}%
\right) ^{1/p}.
\end{eqnarray*}%
Hence, summing over $j$, and using H\"{o}lder inequality%
\begin{eqnarray*}
\Delta &\leq &p2^{p-1}\sum_{j}\omega \left( 
\begin{array}{c}
\tau _{j},\tau _{j+1} \\ 
s,u%
\end{array}%
\right) ^{p-1}\omega \left( 
\begin{array}{c}
\tau _{j},\tau _{j+1} \\ 
t,u%
\end{array}%
\right) \\
&\leq &p2^{p-1}\left( \sum_{j}\omega \left( 
\begin{array}{c}
\tau _{j},\tau _{j+1} \\ 
s,u%
\end{array}%
\right) \right) ^{1-1/p}\left( \sum_{j}\omega \left( 
\begin{array}{c}
\tau _{j},\tau _{j+1} \\ 
t,u%
\end{array}%
\right) \right) ^{1/p} \\
&\leq &p2^{p-1}\omega \left( 
\begin{array}{c}
a,b \\ 
s,u%
\end{array}%
\right) ^{1-1/p}\omega \left( 
\begin{array}{c}
a,b \\ 
t,u%
\end{array}%
\right) ^{1/p}
\end{eqnarray*}%
Interchanging the roles of $R_{1}$and $R_{2},$ we also obtain that 
\begin{equation*}
\Delta \leq p2^{p-1}\omega \left( 
\begin{array}{c}
a,b \\ 
s,u%
\end{array}%
\right) ^{1-1/p}\omega \left( 
\begin{array}{c}
a,b \\ 
t,u%
\end{array}%
\right) ^{1/p},
\end{equation*}%
which concludes the proof.
\end{proof}

\underline{Continuity:} $\omega $ is a map from $\Delta _{T}\times \Delta
_{T}\rightarrow \lbrack 0,\infty )$; the identification of points $\left(
\left( a_{1},a_{2}\right) ,\left( a_{3},a_{4}\right) \right) \in \Delta
_{T}\times \Delta _{T}$ with rectangles in $\left[ 0,T\right] ^{2}$ of the
form $A=\left( 
\begin{array}{c}
a_{1},a_{2} \\ 
a_{3},a_{4}%
\end{array}%
\right) =$ $\left[ a_{1},a_{2}\right] \times \left[ a_{3},a_{4}\right] $ is
pure convention. If $A$ is non-degenerate (i.e. $a_{1}<a_{2},a_{3}<a_{4}$)
and $\left\vert h\right\vert =\max_{i=1}^{4}\left\vert h_{i}\right\vert $
sufficiently small then%
\begin{equation*}
A^{h}:=\left( 
\begin{array}{c}
\left( a_{1}+h_{1}\right) \vee 0,\left( a_{2}+h_{2}\right) \wedge T \\ 
\left( a_{3}+h_{3}\right) \vee 0,\left( a_{4}+h_{4}\right) \wedge T%
\end{array}%
\right)
\end{equation*}
is again a non-degenerate rectangle in $\left[ 0,T\right] ^{2}$. We can then
set for $r>0$, sufficiently small,%
\begin{equation*}
A^{\circ ;r}:=A^{\left( r,-r,r,-r\right) },\,\,\,\bar{A}^{r}:=A^{\left(
-r,r,-r,r\right) }
\end{equation*}%
and note that, whenever $\left\vert h\right\vert $ is small enough to have $%
A^{\circ ;\left\vert h\right\vert }$ well-defined$,$ 
\begin{eqnarray}
A^{\circ ;\left\vert h\right\vert } &\subset &A\subset \bar{A}^{\left\vert
h\right\vert },  \label{AhIncl1} \\
A^{\circ ;\left\vert h\right\vert } &\subset &A^{h}\subset \bar{A}%
^{\left\vert h\right\vert }.  \label{AhIncl2}
\end{eqnarray}

The above definition of $A^{h}$ (and $A^{\circ ;r},\,\bar{A}^{r}$) is easily
extended to degenerate $A$, such that the inclusions (\ref{AhIncl1}),(\ref%
{AhIncl2}) remain valid: For instance, in the case $a_{1}=a_{2}$ we would
replace the first line in the definition of $A^{h}$ by 
\begin{equation*}
\begin{array}{c}
\left( a_{1}+h_{1}\right) \vee 0,\left( a_{2}+h_{2}\right) \wedge T\text{ if 
}h_{1}\leq 0\leq h_{2} \\ 
\left( a_{1}+h_{1}\right) \vee 0,a_{2}\text{ if }h_{1},h_{2}\leq 0 \\ 
a_{1},\left( a_{2}+h_{2}\right) \wedge T\text{ if }h_{1},h_{2}\geq 0 \\ 
a_{1},a_{2}\text{ if }h_{1}\geq 0\geq h_{2}%
\end{array}%
\end{equation*}%
and similarly in the case $a_{3}=a_{4}$. We will prove that, for any
rectangle $A\subset \left[ 0,T\right] ^{2}$,%
\begin{equation*}
\omega \left( A^{h}\right) \rightarrow \omega \left( A\right) \text{ as }%
\left\vert h\right\vert \downarrow 0\text{.}
\end{equation*}%
This end we can and will consider $\left\vert h\right\vert $ is small enough
to have $A^{\circ ;\left\vert h\right\vert }$ (and thus $A^{h},\bar{A}%
^{\left\vert h\right\vert }$) well-defined. By monotonicity of $\omega $, it
follows that%
\begin{equation*}
\omega \left( A^{\circ ;\left\vert h\right\vert }\right) \leq \omega \left(
A^{h}\right) \leq \omega \left( \bar{A}^{\left\vert h\right\vert }\right)
\end{equation*}%
and the limits,%
\begin{eqnarray}
\omega ^{\circ }\left( A\right) &:&=\lim_{r\downarrow 0}\omega \left(
A^{\circ ;r}\right) \leq \omega \left( A\right) ,  \label{oCircLTo} \\
\,\,\,\bar{\omega}\left( A\right) &:&=\lim_{r\downarrow 0}\omega \left( \bar{%
A}^{r}\right) \geq \omega \left( A\right) ,  \notag
\end{eqnarray}%
exist since $\omega \left( A^{\circ ;r}\right) $ [resp. $\omega \left( \bar{A%
}^{r}\right) $] are bounded from above [resp. below] and increasing [resp.
decreasing] as $r\downarrow 0$. It follows that%
\begin{equation*}
\omega ^{\circ }\left( A\right) \leq \underset{\left\vert h\right\vert
\downarrow 0}{\underline{\lim }}\omega \left( A^{h}\right) \leq \underset{%
\left\vert h\right\vert \downarrow 0}{\overline{\lim }}\omega \left(
A^{h}\right) \leq \bar{\omega}\left( A\right) .
\end{equation*}%
The goal is now to show that $\omega ^{\circ }\left( A\right) =\omega \left(
A\right) $ ("inner continuity") and $\bar{\omega}\left( A\right) =\omega
\left( A\right) $ ("outer continuity") since this implies that\ $\underline{%
\lim }\omega \left( A^{h}\right) =\overline{\lim }\omega \left( A^{h}\right)
=\omega \left( A\right) $, which is what we want.\newline
\underline{Inner continuity\textbf{:}}\textbf{\ }We first show that $\omega
^{\circ }$ is super-additive in the sense of definition \ref{Def2DControl}.
To this end, consider $\left\{ R_{i}\right\} \in \mathcal{P}\left( R\right) $%
, some rectangle $R\subset \left[ 0,T\right] ^{2}$. For $r$ small enough,
the rectangles 
\begin{equation*}
\left\{ R_{i}^{0,r}\right\}
\end{equation*}%
are well-defined and essentially disjoint. They can be completed to a
partition of $R^{0,r}$ and hence, by super-additivity of $\omega ,$%
\begin{equation*}
\sum_{i}\omega \left( R_{i}^{0,r}\right) \leq \omega \left( R^{0,r}\right) 
\text{;}
\end{equation*}%
sending $r\downarrow 0$ yields the desired super-addivity of $\omega ^{\circ
}$;%
\begin{equation*}
\sum_{i}\omega ^{\circ }\left( R_{i}\right) \leq \omega ^{\circ }\left(
R\right) .
\end{equation*}
On the other hand, continuity of $f$ on $\left[ 0,T\right] ^{2}$ implies%
\begin{eqnarray*}
\left\vert f\left( A\right) \right\vert ^{p} &\leq &\left\vert f\left(
A^{\circ ,r}\right) \right\vert ^{p}+o\left( 1\right) \\
&\leq &\omega \left( A^{\circ ,r}\right) +o\left( 1\right) \text{ as }r\text{
}\downarrow 0
\end{eqnarray*}%
and hence $\left\vert f\left( A\right) \right\vert ^{p}\leq \omega ^{\circ
}\left( A\right) $, for any rectangle $A\subset \left[ 0,T\right] ^{2}$.
Using super-additivity of $\omega ^{\circ }$ immediately gives%
\begin{equation*}
\omega \left( R\right) \overset{\text{by def.}}{=}\sup_{\Pi \in \mathcal{P}%
\left( R\right) }\sum_{A\in \Pi }\left\vert f\left( A\right) \right\vert
^{p}\leq \omega ^{\circ }\left( R\right) ;
\end{equation*}%
together with (\ref{oCircLTo}) we thus have $\omega \left( R\right) =\omega
^{\circ }\left( R\right) $. Since $R$ was an arbitrary rectangle in $\left[
0,T\right] ^{2}$ inner continuity is proved.

\underline{Outer continuity:} We assume $A\subset \left( 0,T\right) ^{2}$
(i.e. $0<a_{1}\leq a_{2}<T,0<a_{3}\leq a_{4}<T)$ and take $r>0$ small enough
so that%
\begin{equation*}
\bar{A}^{r}=\left( 
\begin{array}{c}
a_{1}-r,a_{2}+r \\ 
a_{3}-r,a_{4}+r%
\end{array}%
\right) ;
\end{equation*}
the general case $A\subset \left[ 0,T\right] ^{2}$ is handled by a (trivial)
adaption of the argument for the remaining cases (i.e. $a_{1}=0$ or $a_{2}=T$
or $a_{3}=0$ or $a_{4}=T$). We first note that \label{individualise}%
\begin{eqnarray*}
\omega \left( \bar{A}^{r}\right) -\omega \left( A\right) &=&\omega \left( 
\begin{array}{c}
a_{1}-r,a_{2}+r \\ 
a_{3}-r,a_{4}+r%
\end{array}%
\right) -\omega \left( 
\begin{array}{c}
a_{1},a_{2} \\ 
a_{3},a_{4}%
\end{array}%
\right) \\
&\leq &\left\vert \omega \left( 
\begin{array}{c}
a_{1}-r,a_{2}+r \\ 
a_{3}-r,a_{4}+r%
\end{array}%
\right) -\omega \left( 
\begin{array}{c}
a_{1}-r,a_{2} \\ 
a_{3}-r,a_{4}+r%
\end{array}%
\right) \right\vert \\
&&+\left\vert \omega \left( 
\begin{array}{c}
a_{1}-r,a_{2} \\ 
a_{3}-r,a_{4}+r%
\end{array}%
\right) -\omega \left( 
\begin{array}{c}
a_{1},a_{2} \\ 
a_{3}-r,a_{4}+r%
\end{array}%
\right) \right\vert \\
&&+\left\vert \omega \left( 
\begin{array}{c}
a_{1},a_{2} \\ 
a_{3}-r,a_{4}+r%
\end{array}%
\right) -\omega \left( 
\begin{array}{c}
a_{1},a_{2} \\ 
a_{3},a_{4}+r%
\end{array}%
\right) \right\vert \\
&&+\left\vert \omega \left( 
\begin{array}{c}
a_{1},a_{2} \\ 
a_{3},a_{4}+r%
\end{array}%
\right) -\omega \left( 
\begin{array}{c}
a_{1},a_{2} \\ 
a_{3},a_{4}%
\end{array}%
\right) \right\vert
\end{eqnarray*}%
Now we use lemma \ref{almostSubAdditive}; with 
\begin{equation*}
\Delta :=\left\vert \omega \left( 
\begin{array}{c}
a_{1}-r,a_{2}+r \\ 
a_{3}-r,a_{4}+r%
\end{array}%
\right) -\omega \left( 
\begin{array}{c}
a_{1}-r,a_{2} \\ 
a_{3}-r,a_{4}+r%
\end{array}%
\right) \right\vert
\end{equation*}%
we have

\begin{eqnarray*}
\Delta &\leq &\omega \left( 
\begin{array}{c}
a_{2},a_{2}+r \\ 
a_{3}-r,a_{4}+r%
\end{array}%
\right) +c\omega \left( \left[ 0,T\right] ^{2}\right) ^{1-1/p}\omega \left( 
\begin{array}{c}
a_{2},a_{2}+r \\ 
a_{3}-r,a_{4}+r%
\end{array}%
\right) ^{1/p} \\
&\leq &\omega \left( 
\begin{array}{c}
a_{2},a_{2}+r \\ 
0,T%
\end{array}%
\right) +c\omega \left( \left[ 0,T\right] ^{2}\right) ^{1-1/p}\omega \left( 
\begin{array}{c}
a_{2},a_{2}+r \\ 
0,T%
\end{array}%
\right) ^{1/p},
\end{eqnarray*}%
\newline
and similar inequalities for the other three terms in our upper estimate on $%
\omega \left( \bar{A}^{r}\right) -\omega \left( A\right) $ above. So it only
remains to prove that for $a\in \left( 0,T\right) $

\begin{equation*}
\omega \left( 
\begin{array}{c}
a,a+r \\ 
0,T%
\end{array}%
\right) ,\text{ }\omega \left( 
\begin{array}{c}
a-r,a \\ 
0,T%
\end{array}%
\right) ,\text{ }\omega \left( 
\begin{array}{c}
0,T \\ 
a,a+r%
\end{array}%
\right) ,\text{ and }\omega \left( 
\begin{array}{c}
0,T \\ 
a-r,a%
\end{array}%
\right)
\end{equation*}%
\newline
converge to $0$ when $r$ tends to $0.$But this is easy; using super-addivity
of $\omega $ and inner-continuity we see that 
\begin{eqnarray*}
\omega \left( 
\begin{array}{c}
a,a+r \\ 
0,T%
\end{array}%
\right) &\leq &\omega \left( 
\begin{array}{c}
a,T \\ 
0,T%
\end{array}%
\right) -\omega \left( 
\begin{array}{c}
a+r,T \\ 
0,T%
\end{array}%
\right) \\
&\rightarrow &0\text{ as }r\downarrow 0\text{.}
\end{eqnarray*}%
Other expressions are handled similarly and our proof of outer continuity is
finished.

\section{Appendix}

\subsection{Young and Young-Towghi discrete inequalities}

\subsubsection{One dimensional case.}

Consider a dissection $D=\left( 0=t_{0},...,t_{n}=T\right) \in \mathcal{D}%
\left( \left[ 0,T\right] \right) .$ We define the "discrete integral"
between $x,y:\left[ 0,T\right] \rightarrow $ $\mathbb{R}$ as

\begin{equation*}
I^{D}=\int_{D}ydx=\sum_{i=1}^{n}y_{t_{i}}x_{t_{i-1},t_{i}}.
\end{equation*}

\begin{lemma}
\label{MainPointinYoung1D}\bigskip Let $p,q\geq 1,$ assume that $\theta
=1/p+1/q>1$. Assume $x,y:\left[ 0,T\right] \rightarrow $ $\mathbb{R}$ are
finite $p$- resp. $q$-variation. Then there exists $t_{i_{0}}\in D\backslash
\left\{ 0,T\right\} $ (equivalently: $i_{0}\in \left\{ 1,\dots ,n-1\right\} $%
) such that 
\begin{equation*}
\left\vert \int_{D}ydx-\int_{D\backslash \left\{ t_{i_{0}}\right\}
}ydx\right\vert \leq \frac{1}{\left( n-1\right) ^{\theta }}\left\vert
x\right\vert _{p\text{-var,}\left[ 0,T\right] }\left\vert y\right\vert _{q%
\text{-var,}\left[ 0,T\right] }
\end{equation*}
\end{lemma}

Iterated removal of points in the dissection, using the above lemma, leads
immediately to Young's maximal inequality which is the heart of the Young's
integral construction.

\begin{theorem}[Young's Maximal Inequality]
Let $p,q\geq 1,$ assume that $\theta =1/p+1/q>1,$ and consider two paths $%
x,y $ from $\left[ 0,T\right] $ into $\mathbb{R}$ of finite $p$-variation
and $q$-variation, with $y_{0}=0.$ Then%
\begin{equation*}
\left\vert \int_{D}ydx\right\vert \leq \left( 1+\zeta \left( \theta \right)
\right) \left\vert x\right\vert _{p\text{-var;}\left[ 0,T\right] }\left\vert
y\right\vert _{q\text{-var;}\left[ 0,T\right] }
\end{equation*}%
and this estimate is uniform over all $D\in \mathcal{D}\left( \left[ 0,T%
\right] \right) $.
\end{theorem}

\begin{proof}
Iterative removal of "$i_{0}$" gives, thanks to lemma \ref%
{MainPointinYoung1D}, 
\begin{eqnarray*}
\left\vert \int_{D}ydx-\int_{\left\{ 0,T\right\} }ydx\right\vert &\leq
&\sum_{n\geq 2}\frac{1}{\left( n-1\right) ^{\theta }}\left\vert x\right\vert
_{p\text{-var,}\left[ 0,T\right] }\left\vert y\right\vert _{q\text{-var,}%
\left[ 0,T\right] } \\
&\leq &\zeta \left( \theta \right) \left\vert x\right\vert _{p\text{-var,}%
\left[ 0,T\right] }\left\vert y\right\vert _{q\text{-var,}\left[ 0,T\right] }
\end{eqnarray*}%
Finally, $\int_{\left\{ 0,T\right\} }ydx=y_{T}x_{0,T}=y_{0,T}x_{0,T}$ since $%
y_{0,T}=y_{T}-y_{0}$ and $y_{0}=0$ and hence 
\begin{equation*}
\left\vert \int_{\left\{ 0,T\right\} }ydx\right\vert =\left\vert
y_{0,T}x_{0,T}\right\vert \leq \left\vert x\right\vert _{p\text{-var,}\left[
0,T\right] }\left\vert y\right\vert _{q\text{-var,}\left[ 0,T\right] }
\end{equation*}%
and we conclude with the triangle inequality.
\end{proof}

\begin{proof}
(\textbf{Lemma }\ref{MainPointinYoung1D}) Observe that, for any $t_{i}\in
D\backslash \left\{ 0,T\right\} $ with $1\leq i\leq n-1$ 
\begin{equation*}
I^{D}-I^{D\backslash \left\{ t_{i}\right\}
}=y_{t_{i},t_{i+1}}x_{t_{i-1},t_{i}}
\end{equation*}%
We pick $t_{i_{0}}$ to make this difference as small as possible:%
\begin{equation*}
\left\vert I^{D}-I^{D\backslash \left\{ t_{i_{0}}\right\} }\right\vert \leq
\left\vert I^{D}-I^{D\backslash \left\{ t_{i}\right\} }\right\vert \text{
for all }i\in \left\{ 1,\dots ,n-1\right\}
\end{equation*}
As an elementary consequence, we have%
\begin{equation*}
\left\vert I^{D}-I^{D\backslash \left\{ t_{i_{0}}\right\} }\right\vert ^{%
\frac{1}{\theta }}\leq \frac{1}{n-1}\sum_{i=1}^{n-1}\left\vert
I^{D}-I^{D\backslash \left\{ t_{i}\right\} }\right\vert ^{1/\theta }.
\end{equation*}

The plan is to get an estimate on $\sum_{i=1}^{n-1}\left\vert
I^{D}-I^{D\backslash \left\{ t_{i}\right\} }\right\vert ^{1/\theta }$
independent of $n$. In fact, we shall see that 
\begin{equation}
\sum_{i=1}^{n-1}\left\vert I^{D}-I^{D\backslash \left\{ t_{i}\right\}
}\right\vert ^{1/\theta }\leq \left\vert x\right\vert _{p\text{-var,}\left[
0,T\right] }^{1/\theta }\left\vert y\right\vert _{q\text{-var,}\left[ 0,T%
\right] }^{1/\theta }  \label{1DLemma}
\end{equation}
and the desired estimate 
\begin{equation*}
\left\vert I^{D}-I^{D\backslash \left\{ t_{i_{0}}\right\} }\right\vert \leq
\left( \frac{1}{n-1}\right) ^{\theta }\left\vert x\right\vert _{p\text{-var,}%
\left[ 0,T\right] }\left\vert y\right\vert _{q\text{-var,}\left[ 0,T\right] }
\end{equation*}%
follows. It remains to establish (\ref{1DLemma}); thanks to H\"{o}lder's
inequality, using $1/\left( q\theta \right) +1/\left( p\theta \right) =1,$%
\begin{eqnarray*}
\sum_{i=1}^{n-1}\left\vert I^{D}-I^{D\backslash \left\{ t_{i}\right\}
}\right\vert ^{1/\theta } &=&\left( \sum_{i=1}^{n-1}\left\vert
y_{t_{i},t_{i+1}}\right\vert ^{1/\theta }\left\vert
x_{t_{i-1},t_{i}}\right\vert ^{1/\theta }\right) ^{\theta } \\
&\leq &\left( \sum_{i=1}^{m-1}\left\vert y_{t_{i},t_{i+1}}\right\vert
^{q}\right) ^{\frac{1}{q\theta }}\left( \sum_{i=1}^{n-1}\left\vert
x_{t_{i-1},t_{i}}\right\vert ^{p}\right) ^{\frac{1}{p\theta }} \\
&\leq &\left\vert x\right\vert _{p\text{-var,}\left[ 0,T\right] }^{1/\theta
}\left\vert y\right\vert _{q\text{-var,}\left[ 0,T\right] }^{1/\theta }.
\end{eqnarray*}%
and we are done.
\end{proof}

\subsubsection{Young-Towghi maximal inequality\ (2D)\label%
{AppendixYoungTowghiMaxInequ}}

We now consider the two-dimensional case. To this end, fix two dissections $%
D=\left( 0=t_{0},...,t_{n}=T\right) $ and $D^{\prime }=\left(
0=t_{0}^{\prime },...,t_{m}^{\prime }=T\right) ,$and define the discrete
integral between $x,y:\left[ 0,T\right] ^{2}\rightarrow $ $\mathbb{R}$ as 
\begin{equation}
I^{D,D^{\prime }}=\int_{D\times D^{\prime }}ydx:=\sum_{i}\sum_{j}y\left( 
\begin{array}{c}
t_{i} \\ 
t_{j}^{\prime }%
\end{array}%
\right) x\left( 
\begin{array}{c}
t_{i-1},t_{i} \\ 
t_{j-1}^{\prime },t_{j}^{\prime }%
\end{array}%
\right) .  \label{IntDDprimeNotation}
\end{equation}

\begin{lemma}
\label{MainPointinYoung2D}Let $p,q\geq 1,$ assume that $\theta =1/p+1/q>1$.
Assume $x,y:\left[ 0,T\right] ^{2}\rightarrow $ $\mathbb{R}$ are finite $p$-
resp. $q$-variation. Then there exists $t_{i_{0}}\in D\backslash \left\{
0,T\right\} $ (equivalently: $i_{0}\in \left\{ 1,\dots ,n-1\right\} $ such
that for every $\alpha \in \left( 1,\theta \right) $, 
\begin{equation*}
\left\vert \int_{D\times D^{\prime }}dx-\int_{D\backslash \left\{
t_{i_{0}}\right\} \times D^{\prime }}ydx\right\vert \leq \left( \frac{1}{n-1}%
\right) ^{\alpha }\left( 1+\zeta \left( \frac{\theta }{\alpha }\right)
\right) ^{\alpha }V_{p}\left( x;\left[ 0,T\right] ^{2}\right) V_{q}\left( y;%
\left[ 0,T\right] ^{2}\right)
\end{equation*}
\end{lemma}

Iterative removal of "$i_{0}$" leads to Young-Towghi's maximal inequality.

\begin{theorem}[Young-Towghi Maximal Inequality]
\label{YTMI}Let $p,q\geq 1,$ assume that $\theta =1/p+1/q>1,$ and consider $%
x,y:\left[ 0,T\right] ^{2}\rightarrow $ $\mathbb{R}$ of finite $p$- resp. $q$%
-variation and $y\left( 0,\cdot \right) =y\left( \cdot ,0\right) =0$. Then,
for every $\alpha \in \left( 1,\theta \right) $,%
\begin{equation*}
\left\vert \int_{D\times D^{\prime }}ydx\right\vert \leq \left[ \left(
1+\zeta \left( \frac{\theta }{\alpha }\right) \right) ^{\alpha }\zeta \left(
\alpha \right) +\left( 1+\zeta \left( \theta \right) \right) \right]
V_{p}\left( x;\left[ 0,T\right] ^{2}\right) V_{q}\left( y;\left[ 0,T\right]
^{2}\right)
\end{equation*}%
and this estimate is uniform over all $D,D^{\prime }\in \mathcal{D}\left( %
\left[ 0,T\right] \right) $
\end{theorem}

\begin{proof}
Iterative removal of "$i_{0}$" gives 
\begin{eqnarray*}
\left\vert \int_{D\times D^{\prime }}ydx-\int_{\left\{ 0,T\right\} \times
D^{\prime }}ydx\right\vert &\leq &\sum_{n\geq 2}\left( \frac{1}{n-1}\right)
^{\alpha }\left( 1+\zeta \left( \frac{\theta }{\alpha }\right) \right)
^{\alpha }V_{p}\left( x;\left[ 0,T\right] ^{2}\right) V_{q}\left( y;\left[
0,T\right] ^{2}\right) \\
&\leq &\zeta \left( \alpha \right) \left( 1+\zeta \left( \frac{\theta }{%
\alpha }\right) \right) ^{\alpha }V_{p}\left( x;\left[ 0,T\right]
^{2}\right) V_{q}\left( y;\left[ 0,T\right] ^{2}\right) .
\end{eqnarray*}%
It only remains to bound%
\begin{equation*}
\int_{\left\{ 0,T\right\} \times D^{\prime }}ydx=\sum_{j}y\left( 
\begin{array}{c}
T \\ 
t_{j}^{\prime }%
\end{array}%
\right) x\left( 
\begin{array}{c}
0,T \\ 
t_{j-1}^{\prime },t_{j}^{\prime }%
\end{array}%
\right) =\int_{D^{\prime }}y\left( 
\begin{array}{c}
0,T \\ 
.%
\end{array}%
\right) dx\left( 
\begin{array}{c}
0,T \\ 
.%
\end{array}%
\right)
\end{equation*}%
where we used $y\left( 
\begin{array}{c}
0 \\ 
\cdot%
\end{array}%
\right) =0$ in the last equality. From Young's 1D maximal inequality, we have%
\begin{eqnarray*}
\left\vert \int_{\left\{ 0,T\right\} \times D^{\prime }}ydx\right\vert &\leq
&\left( 1+\zeta \left( \theta \right) \right) \left\vert y\left( 
\begin{array}{c}
0,T \\ 
0,.%
\end{array}%
\right) \right\vert _{q\text{-var,}\left[ 0,T\right] }\left\vert x\left( 
\begin{array}{c}
0,T \\ 
0,.%
\end{array}%
\right) \right\vert _{p\text{-var,}\left[ 0,T\right] } \\
&\leq &\left( 1+\zeta \left( \theta \right) \right) V_{p}\left( x;\left[ 0,T%
\right] ^{2}\right) V_{q}\left( y;\left[ 0,T\right] ^{2}\right)
\end{eqnarray*}%
The triangle inequality allows us to conclude.
\end{proof}

\begin{proof}
(\textbf{Lemma }\ref{MainPointinYoung2D}) Observe that, for any $t_{i}\in
D\backslash \left\{ 0,T\right\} $ 
\begin{eqnarray*}
I^{D,D^{\prime }}-I^{D\backslash \left\{ t_{i}\right\} ,D^{\prime }}
&=&\int_{D^{\prime }}y\left( 
\begin{array}{c}
t_{i},t_{i+1} \\ 
\cdot%
\end{array}%
\right) x\left( 
\begin{array}{c}
t_{i-1},t_{i} \\ 
\cdot%
\end{array}%
\right) \\
&=&\int_{D^{\prime }}y\left( 
\begin{array}{c}
t_{i},t_{i+1} \\ 
0,\cdot%
\end{array}%
\right) x\left( 
\begin{array}{c}
t_{i-1},t_{i} \\ 
\cdot%
\end{array}%
\right)
\end{eqnarray*}%
where we used $y\left( 
\begin{array}{c}
\cdot \\ 
0%
\end{array}%
\right) =0$. We pick $t_{i_{0}}$ to make this difference as small as
possible:%
\begin{equation*}
\left\vert I^{D,D^{\prime }}-I^{D\backslash \left\{ t_{i_{0}}\right\}
,D^{\prime }}\right\vert \leq \left\vert I^{D,D^{\prime }}-I^{D\backslash
\left\{ t_{i}\right\} ,D^{\prime }}\right\vert \text{ for all }i\in \left\{
1,\dots ,n-1\right\}
\end{equation*}%
As an elementary consequence, 
\begin{equation}
\left\vert I^{D,D^{\prime }}-I^{D\backslash \left\{ t_{i_{0}}\right\}
,D^{\prime }}\right\vert ^{1/\alpha }\leq \frac{1}{n-1}\sum_{i=1}^{n-1}\left%
\vert I^{D,D^{\prime }}-I^{D\backslash \left\{ t_{i}\right\} ,D^{\prime
}}\right\vert ^{1/\alpha }.  \label{BoundQ0o}
\end{equation}%
The plan is to get an estimate on $\sum_{i=1}^{n-1}\left\vert I^{D,D^{\prime
}}-I^{D\backslash \left\{ t_{i}\right\} ,D^{\prime }}\right\vert ^{1/\alpha
} $ independent of $n$ and uniformly in $D^{\prime }\in \mathcal{D}\left( %
\left[ 0,T\right] \right) $; in fact, we shall see that 
\begin{equation}
\Delta ^{D,D^{\prime }}:=\sum_{i=1}^{n-1}\left\vert I^{D,D^{\prime
}}-I^{D\backslash \left\{ t_{i}\right\} ,D^{\prime }}\right\vert ^{1/\alpha
}\leq cV_{p}\left( x;\left[ 0,T\right] ^{2}\right) ^{1/\alpha }V_{q}\left( y;%
\left[ 0,T\right] ^{2}\right) ^{1/\alpha }  \label{2DLemma}
\end{equation}%
with $c=1+\zeta \left( \frac{\theta }{\alpha }\right) $ and the desired
estimate 
\begin{equation*}
\left\vert I^{D}-I^{D\backslash \left\{ t_{i_{0}}\right\} }\right\vert \leq
\left( \frac{c}{n-1}\right) ^{\alpha }V_{p}\left( x;\left[ 0,T\right]
^{2}\right) V_{q}\left( y;\left[ 0,T\right] ^{2}\right)
\end{equation*}%
follows. It remains to establish (\ref{2DLemma}); to this end we consider
the removal of $t_{j}^{\prime }\in D^{\prime }\backslash \left\{ 0,T\right\} 
$ from $D^{\prime }$ and note that 
\begin{equation*}
\left( I^{D,D^{\prime }}-I^{D\backslash \left\{ t_{i}\right\} ,D^{\prime
}}\right) -\left( I^{D,D^{\prime }\backslash \left\{ t_{j}^{\prime }\right\}
}-I^{D\backslash \left\{ t_{i}\right\} ,D^{\prime }\backslash \left\{
t_{j}^{\prime }\right\} }\right) =y\left( 
\begin{array}{c}
t_{i},t_{i+1} \\ 
t_{j}^{\prime },t_{j+1}^{\prime }%
\end{array}%
\right) x\left( 
\begin{array}{c}
t_{i-1},t_{i} \\ 
t_{j-1}^{\prime },t_{j}^{\prime }%
\end{array}%
\right)
\end{equation*}%
Using the elementary inequality $\left\vert a\right\vert ^{1/\alpha
}-\left\vert b\right\vert ^{1/\alpha }\leq \left\vert a-b\right\vert
^{1/\alpha }$ valid for $a,b\in \mathbb{R}$ and $\alpha \geq 1$ we have%
\begin{equation*}
\left. 
\begin{array}{l}
\left\vert I^{D,D^{\prime }}-I^{D\backslash \left\{ t_{i}\right\} ,D^{\prime
}}\right\vert ^{1/\alpha }-\left\vert I^{D,D^{\prime }\backslash \left\{
t_{j}^{\prime }\right\} }-I^{D\backslash \left\{ t_{i}\right\} ,D^{\prime
}\backslash \left\{ t_{j}^{\prime }\right\} }\right\vert ^{1/\alpha } \\ 
\text{\text{ }}\leq \left\vert \left( I^{D,D^{\prime }}-I^{D\backslash
\left\{ t_{i}\right\} ,D^{\prime }}\right) -\left( I^{D,D^{\prime
}\backslash \left\{ t_{j}^{\prime }\right\} }-I^{D\backslash \left\{
t_{i}\right\} ,D^{\prime }\backslash \left\{ t_{j}^{\prime }\right\}
}\right) \right\vert ^{1/\alpha }.%
\end{array}%
\right.
\end{equation*}%
Hence, summing over $i$, we get 
\begin{eqnarray}
&&\Delta ^{D,D^{\prime }}-\Delta ^{D,D^{\prime }\backslash \left\{
t_{j}^{\prime }\right\} }  \notag \\
&\leq &\sum_{i=1}^{n-1}\left\vert \left( I^{D,D^{\prime }}-I^{D\backslash
\left\{ t_{i}\right\} ,D^{\prime }}\right) -\left( I^{D,D^{\prime
}\backslash \left\{ t_{j}^{\prime }\right\} }-I^{D\backslash \left\{
t_{i}\right\} ,D^{\prime }\backslash \left\{ t_{j}^{\prime }\right\}
}\right) \right\vert ^{1/\alpha }  \notag \\
&=&\sum_{i=1}^{n-1}\left\vert y\left( 
\begin{array}{c}
t_{i},t_{i+1} \\ 
t_{j}^{\prime },t_{j+1}^{\prime }%
\end{array}%
\right) \right\vert ^{1/\alpha }\left\vert x\left( 
\begin{array}{c}
t_{i-1},t_{i} \\ 
t_{j-1}^{\prime },t_{j}^{\prime }%
\end{array}%
\right) \right\vert ^{1/\alpha }  \label{long} \\
&\leq &\left( \sum_{i=1}^{n-1}\left\vert y\left( 
\begin{array}{c}
t_{i},t_{i+1} \\ 
t_{j}^{\prime },t_{j+1}^{\prime }%
\end{array}%
\right) \right\vert ^{\theta q/\alpha }\right) ^{\frac{1}{\theta q}}\left(
\sum_{i=1}^{n-1}\left\vert x\left( 
\begin{array}{c}
t_{i-1},t_{i} \\ 
t_{j-1}^{\prime },t_{j}^{\prime }%
\end{array}%
\right) \right\vert ^{\theta p/\alpha }\right) ^{\frac{1}{\theta p}}  \notag
\\
&\leq &\left( \sum_{i=1}^{n-1}\left\vert y\left( 
\begin{array}{c}
t_{i},t_{i+1} \\ 
t_{j}^{\prime },t_{j+1}^{\prime }%
\end{array}%
\right) \right\vert ^{q}\right) ^{\frac{1}{\alpha q}}\left(
\sum_{i=1}^{n-1}\left\vert x\left( 
\begin{array}{c}
t_{i-1},t_{i} \\ 
t_{j-1}^{\prime },t_{j}^{\prime }%
\end{array}%
\right) \right\vert ^{p}\right) ^{\frac{1}{\alpha p}};  \notag
\end{eqnarray}%
in the last step we used that the $\ell ^{\theta p/\alpha }$ norm on $%
\mathbb{R}^{n-1}$ is dominated by the $\ell ^{p}$ norm (because $\theta
p/\alpha >p$). It follows that%
\begin{equation}
\Delta ^{D,D^{\prime }}-\Delta ^{D,D^{\prime }\backslash \left\{
t_{j}^{\prime }\right\} }\leq Y_{j}^{1/\alpha }X_{j}^{1/\alpha }
\label{DeltaYX}
\end{equation}%
where%
\begin{equation*}
Y_{j}:=\left( \sum_{i=1}^{n-1}\left\vert y\left( 
\begin{array}{c}
t_{i},t_{i+1} \\ 
t_{j}^{\prime },t_{j+1}^{\prime }%
\end{array}%
\right) \right\vert ^{q}\right) ^{\frac{1}{q}},\,\,X_{j}:=\left(
\sum_{i=1}^{n-1}\left\vert x\left( 
\begin{array}{c}
t_{i-1},t_{i} \\ 
t_{j-1}^{\prime },t_{j}^{\prime }%
\end{array}%
\right) \right\vert ^{p}\right) ^{\frac{1}{p}}
\end{equation*}%
We pick $t_{j_{0}}^{\prime }\in D^{\prime }\backslash \left\{ 0,T\right\} $
(i.e. $1\leq j_{0}\leq m-1$) to make this difference as small as possible,%
\begin{equation*}
\Delta ^{D,D^{\prime }}-\Delta ^{D,D^{\prime }\backslash \left\{
t_{j_{0}}^{\prime }\right\} }\leq \Delta ^{D,D^{\prime }}-\Delta
^{D,D^{\prime }\backslash \left\{ t_{j}^{\prime }\right\} }\text{ for all }%
j\in \left\{ 1,\dots ,m-1\right\} ;
\end{equation*}
we shall see below that%
\begin{equation}
\left\vert \Delta ^{D,D^{\prime }}-\Delta ^{D,D^{\prime }\backslash \left\{
t_{j_{0}}^{\prime }\right\} }\right\vert \leq \left( \frac{1}{m-1}\right) ^{%
\frac{\theta }{\alpha }}V_{p}\left( x;\left[ 0,T\right] ^{2}\right)
^{1/\alpha }V_{q}\left( y;\left[ 0,T\right] ^{2}\right) ^{1/\alpha };
\label{stilltodo}
\end{equation}%
iterated removal of "$j_{0}$" yields%
\begin{equation*}
\Delta ^{D,D^{\prime }}\leq \Delta ^{D,\left\{ 0,T\right\} }+\zeta \left( 
\frac{\theta }{\alpha }\right) V_{p}\left( x,\left[ 0,T\right] ^{2}\right)
^{1/\alpha }V_{q}\left( y,\left[ 0,T\right] ^{2}\right) ^{1/\alpha };
\end{equation*}%
as in (\ref{long}) we estimate 
\begin{equation*}
\Delta ^{D,\left\{ 0,T\right\} }=\sum_{i=1}^{n-1}\left\vert y\left( 
\begin{array}{c}
t_{i},t_{i+1} \\ 
0,T%
\end{array}%
\right) x\left( 
\begin{array}{c}
t_{i-1},t_{i} \\ 
0,T%
\end{array}%
\right) \right\vert ^{1/\alpha }\leq \dots \leq V_{p}\left( x,\left[ 0,T%
\right] ^{2}\right) ^{1/\alpha }V_{q}\left( y,\left[ 0,T\right] ^{2}\right)
^{1/\alpha }
\end{equation*}%
and (\ref{2DLemma}) follows, as desired. The only thing left is to establish
(\ref{stilltodo}). Using (\ref{DeltaYX}) we can write%
\begin{eqnarray*}
\Delta ^{D,D^{\prime }}-\Delta ^{D,D^{\prime }\backslash \left\{
t_{j_{0}}^{\prime }\right\} } &\leq &\left( \prod_{j=1}^{m-1}\Delta
^{D,D^{\prime }}-\Delta ^{D,D^{\prime }\backslash \left\{ t_{j}^{\prime
}\right\} }\right) ^{\frac{1}{m-1}} \\
&\leq &\left( \prod_{j=1}^{m-1}X_{j}^{1/.\alpha }Y_{j}^{1/\alpha }\right) ^{%
\frac{1}{m-1}} \\
&=&\left( \prod_{j=1}^{m-1}X_{j}^{p}\right) ^{\frac{1}{m-1}\frac{1}{\alpha p}%
}\left( \prod_{j=1}^{m-1}Y_{j}^{q}\right) ^{\frac{1}{m-1}\frac{1}{\alpha q}}.
\end{eqnarray*}

Using the geometric/arithmetic inequality, we obtain%
\begin{eqnarray*}
\left( \prod_{j=1}^{m-1}X_{j}^{p}\right) ^{\frac{1}{m-1}\frac{1}{\alpha p}}
&\leq &\left( \frac{1}{m-1}\sum_{j=1}^{m-1}X_{j}^{p}\right) ^{\frac{1}{%
\alpha p}} \\
&\leq &\left( \frac{1}{m-1}\right) ^{\frac{1}{\alpha p}}\left(
\sum_{j=1}^{m-1}\sum_{i=1}^{n-1}\left\vert x\left( 
\begin{array}{c}
t_{i-1},t_{i} \\ 
t_{j-1}^{\prime },t_{j}^{\prime }%
\end{array}%
\right) \right\vert ^{p}\right) ^{\frac{1}{\alpha p}} \\
&\leq &\left( \frac{1}{m-1}\right) ^{\frac{1}{\alpha p}}V_{p}\left( x,\left[
0,T\right] ^{2}\right) ^{1/\alpha }.
\end{eqnarray*}%
and, similarly, 
\begin{equation*}
\left( \prod_{j=1}^{m-1}Y_{j}^{q}\right) ^{\frac{1}{m-1}\frac{1}{\alpha q}%
}\leq \left( \frac{1}{m-1}\right) ^{\frac{1}{\alpha q}}V_{q}\left( y,\left[
0,T\right] ^{2}\right) ^{1/\alpha }.
\end{equation*}

Using $\frac{1}{\alpha p}+\frac{1}{\alpha q}=\frac{\theta }{\alpha }$, we
thus arrive at%
\begin{equation*}
\Delta ^{D,D^{\prime }}-\Delta ^{D,D^{\prime }\backslash \left\{
t_{j_{0}}^{\prime }\right\} }\leq \left( \frac{1}{m-1}\right) ^{\frac{\theta 
}{\alpha }}V_{p}\left( x,\left[ 0,T\right] ^{2}\right) ^{1/\alpha
}V_{q}\left( y,\left[ 0,T\right] ^{2}\right) ^{1/\alpha }
\end{equation*}%
which is precisely the claimed estimate (\ref{stilltodo}).
\end{proof}

\end{document}